\documentclass[reqno]{amsart}
\usepackage{amsmath}
\usepackage{amsthm, amscd, amssymb, amsfonts, amsbsy}
\usepackage[usenames, dvipsnames]{color}
\usepackage{enumerate}
\usepackage{hyperref}
\usepackage{mathrsfs}
\usepackage{verbatim}
\usepackage{todonotes}

\numberwithin{equation}{section}

\theoremstyle{plain}
\newtheorem{theorem}{Theorem}[section]
\newtheorem{lemma}[theorem]{Lemma}
\newtheorem{corollary}[theorem]{Corollary}
\newtheorem{proposition}[theorem]{Proposition}

\theoremstyle{definition}
\newtheorem{definition}[theorem]{Definition}
\newtheorem{assumption}[theorem]{Assumption}

\theoremstyle{remark}
\newtheorem{remark}[theorem]{Remark}

\makeatletter
\def\dashint{\operatorname%
{\,\,\text{\bf--}\kern-.98em\DOTSI\intop\ilimits@\!\!}}
\makeatother

\def\bR{\mathbb{R}}

\def\cB{\mathcal{B}}

\def\cD{\mathcal{D}}
\def\cL{\mathcal{L}}
\def\cN{\mathcal{N}}

\begin{document}
\title[Green functions in two dimensions]{Green functions of mixed boundary value problems for stationary Stokes systems in two dimensions}

\author[J. Choi]{Jongkeun Choi}
\address[J. Choi]{
Department of Mathematics Education, Pusan National University,  Busan, 46241, Republic of Korea}

\email{jongkeun\_choi@pusan.ac.kr}

\thanks{J. Choi was supported by the National Research Foundation of Korea(NRF) under agreement NRF-2022R1F1A1074461}

\author[M. Yang]{Minsuk Yang}
\address[M. Yang]{Department of Mathematics, Yonsei University, Seoul 03722, Republic of Korea}
\email{m.yang@yonsei.ac.kr}

\thanks{M. Yang was supported by the National Research Foundation of Korea(NRF) under agreement NRF-2021R1A2C4002840}

\subjclass[2010]{35J08, 35J57, 76D07}
\keywords{Green function, Stokes system, mixed boundary value problem, measurable coefficients, Reifenberg flat domain}

\begin{abstract}
We establish the existence, uniqueness, and various estimates for Green functions of mixed Dirichlet-conormal derivative problems for the stationary Stokes system with measurable coefficients in a two-dimensional Reifenberg flat domain with a rough separation. 
\end{abstract}

\maketitle

%\tableofcontents

%========================================
\section{Introduction}	
%========================================

Let $\Omega$ be a bounded domain in $\bR^2$.
The boundary of $\Omega$, denoted by $\partial \Omega$, is divided into two open components $\cD$ and $\cN$ separated by $\Gamma$.
That is,
$$
\partial \Omega=\overline{\cD}\cup \cN, \quad \cD\cap \cN=\emptyset, \quad \Gamma=\overline{\cD}\cap \overline{\cN}.
$$
We consider mixed Dirichlet-conormal derivative problems for the stationary Stokes system with variable coefficients
\begin{equation}		\label{230216_eq1}
\begin{cases}
\operatorname{div} u=g &\text{in }\, \Omega,\\
\cL u+\nabla p=f+D_\alpha f^\alpha &\text{in }\, \Omega,\\
\cB u+p \nu=\nu_\alpha f^\alpha  & \text{on }\, \cN,\\
u=0 & \text{on }\, \cD.
\end{cases}
\end{equation}
The elliptic operator $\cL$ and its associated conormal derivative operator $\cB$ are given by 
$$
\cL u=D_\alpha(A^{\alpha\beta}D_\beta u), \quad \cB u=\nu_\alpha A^{\alpha\beta} D_\beta u,
$$
where the Einstein summation convention on repeated indices is used.
The coefficients $A^{\alpha\beta}$, as functions of $x\in \bR^2$, need only be measurable satisfying the strong ellipticity condition.
The Stokes system with such irregular coefficients arises naturally in mathematical fluid dynamics, particularly when dealing with inhomogeneous fluids with density-dependent viscosity and Stokes flow over composite materials.
See, for instance,  \cite{MR2663713, MR4440019, MR0425391, MR1422251, MN1987}.
Also, the system with  the mixed boundary condition is connected to the study of incompressible medium or fluid flow on an artificial boundary, such as a canal exit. 
For applications, see \cite{MR1676350, 2005Kohr, MR2182091} and references therein.

In this paper, we are concerned with a Green function (for the flow velocity) of the mixed problem \eqref{230216_eq1}, which is defined by a pair $(G, \Pi)$ satisfying
$$
\begin{cases}
\operatorname{div} G(\cdot,y)=0 &\text{in }\, \Omega,\\
\cL G(\cdot,y)+\nabla \Pi(\cdot,y)=-\delta_y I &\text{in }\, \Omega,\\
\cB G(\cdot,y)+\nu\Pi(\cdot, y) =0 & \text{on }\, \cN,\\
u=0 & \text{on }\, \cD.
\end{cases}
$$
See Definition \ref{D1} for a precise definition of the Green function. 
We prove that the Green function $(G, \Pi)$ exists and has the logarithmic pointwise bound
$$
|G(x,y)|\le C\log \bigg(\frac{\operatorname{diam}(\Omega)}{|x-y|}\bigg)+C
$$
in a Reifenberg flat domain $\Omega\subset\bR^2$, under the assumption that the Dirichlet part $\cD$ accounts for a certain proportion near the separation $\Gamma$ at every scale.
We also establish the representation formula
$$
u(y)=-\int_\Omega G(y, \cdot) f\,dx+\int_\Omega D_\alpha G(y, \cdot) f^\alpha \,dx
$$
whenever $(u,p)$ is a weak solution of \eqref{230216_eq1} with $g\equiv 0$.
For further details, see Theorem \ref{M1} and Remark \ref{230216_rmk2}.
Note that Lipschitz domains with small Lipschitz constants are Reifenberg flat.
Hence, our main theorem (Theorem \ref{M1}) holds for all Lipschitz domains with small Lipschitz constants.
In fact, as far as the Lipschitz domains are concerned, our argument is irrelevant to the flatness of the Lipschitz boundary.
Thus, the theorem still holds for all Lipschitz domains with a bounded Lipschitz constant.
Moreover, in this case, the assumption on $\cD$ can be relaxed to the so-called Ahlfors-David condition, which means the surface measure of $\cD\cap B_R(x_0)$ is comparable to $R$ when $x\in \cD$.
See Remark \ref{230216_rmk3} for more details.

Green functions for the Stokes system have been studied extensively in the literature.
With respect to the classical Stokes system 
$$
\Delta u+\nabla p=f,
$$
we refer the reader to \cite{MR2465713, MR0975121, MR2718661, MR0254401, MR0725151, MR0734895, MR2763343}  and references therein.
In these papers, the existence and estimates for fundamental solutions and Green functions of pure Dirichlet/Neumann problems were studied.
We also refer to \cite{MR2182091, MR3320459} for Green functions of  mixed  Dirichlet-Neumann problems.
In \cite{MR2182091}, Maz'ya and Rossmann established pointwise estimates for the Green function and its derivatives in a 3D polyhedral cone, on faces of which 
the different boundary conditions are prescribed. 
%Such Green function estimates were utilized in \cite{MR2321139, MR2563641} to develop $L^q$-regularity theory for the mixed problems for stationary Stokes and Navier-Stokes systems in a polyhedral domain.
In \cite{MR3320459}, Ott, Kim, and Brown constructed the Green function having a logarithmic pointwise bound in  a 2D Lipschitz domain with the Ahlfors-David condition.
Their argument relies on the solvability of the system with data lying in the dual of Lorentz-Sobolev spaces, and as commented in that paper, it may be applicable to a more general setting with irregular coefficients and domains.

Recently,  in  \cite{MR3906316,CK2023} the Green functions for the Stokes system with measurable coefficients  in 2D domains were studied.
The authors in \cite{MR3906316} established the existence and pointwise estimates for the Green function of the pure Dirichlet problem  in John domains. 
They also discussed the global pointwise bounds of the derivatives under regularity assumptions on the coefficients and the boundary of the domain.
The corresponding results for the pure conormal derivative problem were achieved in \cite{CK2023}.
See also \cite{MR3877495, MR3693868, MR3959936} for earlier works on three or higher-dimensional Green functions of pure boundary value problems with BMO (bounded mean oscillations) or periodically oscillating coefficients.

To prove the main result of this paper, we employ the approach used in \cite{CK2023}, differing from the methodology outlined in  \cite{MR3320459}.
The key ingredient lies in establishing uniform estimates for a family of approximated Green functions, relying  on $L^q$-estimates, where $q$ is  close to $2$, and Sobolev embeddings.
In this paper, we refine the approach for the $L^q$-estimates and embeddings since dealing with the mixed Dirichlet-conormal boundary condition makes the argument more involved.

The remainder of this paper is  organized as follows.
In Section \ref{S2}, we fix our notation, introduce assumptions on the domain, and then state our main results.
We provide some auxiliary results in Section \ref{S3} and prove the main theorem in Section \ref{S4}.

%========================================
\section{Preliminaries and main results}	\label{S2}
%========================================

We first fix some notation and standard definitions used throughout the paper.
Let $\Omega$ be a bounded domain in $\bR^2$.
%We assume that $\partial \Omega$ is divided into two disjoint nonempty components $\cD$ and $\cN$ separated by $\Gamma$.
%That is, 
%$$
%\partial \Omega=\cD\cup \cN, \quad \cD\cap \cN=\emptyset, \quad \Gamma=\overline{\cD}\cap \overline{\cN},
%$$
%and $\cD$ is open (relative to $\partial \Omega$).
%In Assumptions \ref{A1} and \ref{A2} below, we will impose certain regularity assumptions on $\partial \Omega$ and $\Gamma$.
For any $x\in \bR^2$ and $R>0$, we write $\Omega_R(x)=\Omega\cap B_R(x)$, where $B_R(x)$ is the usual Euclidean disk of radius $R$ centered at $x$.
For a function $u$ in $\Omega$, we set 
$$
\|u\|_{L^{2, \infty}(\Omega)}=\sup_{t>0} t\big|\{x\in \Omega: |u(x)|>t\}\big|^{1/2}
$$
and 
$$
(u)_{\Omega}=\dashint_{\Omega} u\,dx=\frac{1}{|\Omega|}\int_\Omega u\,dx,
$$
where $|\Omega|$ is the $2$-dimensional Lebesgue measure of $\Omega$.
For $q\in [1, \infty]$, we denote by $W^{1,q}(\Omega)$ the usual Sobolev space and by $W^{1,q}_{\cD}(\Omega)$  the completion of $C^\infty_{\cD}(\Omega)$ in $W^{1, q}(\Omega)$, where $C^\infty_{\cD}(\Omega)$ is the set of all infinitely differentiable functions in $\bR^2$ having a compact support in $\overline{\Omega}$ and vanishing in a neighborhood of $\overline{\cD}$. 
Similarly, we define $W^{1,q}_{\cN}(\Omega)$.

By, for instance, $u\in W^{1,q}(\Omega)^2$ we mean that $u=(u_1, u_2)^{\top}$ and $u_1,u_2\in W^{1,q}(\Omega)$.
As a superscript, we also use $2\times 2$ (resp. $1\times 2$) in place of $2$ to denote a space for $2\times 2$ (resp. $1\times 2$) matrix-valued functions.

%========================================
\subsection{Mixed boundary value problem and Green function}	\label{S2_1}
%========================================

Let $\cL$ be a strongly elliptic operator of the form 
$$
\cL u=D_\alpha (A^{\alpha\beta} D_\beta u),
$$
where the coefficients $A^{\alpha\beta}=A^{\alpha\beta}(x)$ are $2\times 2$ matrix-valued functions in $\bR^2$ satisfying the strong ellipticity condition,
that is, there exists $\lambda\in (0, 1]$ such  
$$
|A^{\alpha\beta}|\le \lambda^{-1}, \quad \lambda  |\xi|^2\le A^{\alpha\beta}_{ij} \xi^\beta_j\xi^\alpha_i
$$
for any $\xi=(\xi^\alpha_i)\in \bR^{2\times 2}$.
We denote by $\cB u=\nu_\alpha A^{\alpha\beta} D_\beta u$ the conormal derivative of $u$ on $\partial\Omega$ associated with $\cL$, where $\nu=(\nu_1, \nu_2)^{\top}$ is the outward unit normal to $\partial \Omega$.
The adjoint operator $\cL^*$ and its associated conormal derivative operator $\cB^*$ are defined by 
$$
\cL^* u=D_\alpha \big((A^{\beta\alpha})^{\top} D_\beta u\big), \quad \cB^* u=\nu_\alpha (A^{\beta\alpha})^{\top} D_\beta u.
$$
Let $q,q_1\in (1, \infty)$ with $q_1\ge 2q/(2+q)$.
For $f\in L^{q_1}(\Omega)^2$, $f^\alpha\in L^q(\Omega)^2$, and $g\in L^q(\Omega)$, we say that $(u, p)\in W^{1,q}_{\cD}(\Omega)^2\times L^q(\Omega)$ is a weak solution of the problem 
$$
\begin{cases}
\operatorname{div} u=g &\text{in }\, \Omega,\\
\cL u+\nabla p=f+D_\alpha f^\alpha &\text{in }\, \Omega,\\
\cB u+p \nu=\nu_\alpha f^\alpha  & \text{on }\, \cN,\\
u=0 & \text{on }\, \cD,
\end{cases}
$$
if $\operatorname{div} u=g$ a.e. in $\Omega$ and 
$$
\int_\Omega A^{\alpha\beta} D_\beta u\cdot D_\alpha \varphi\,dx+\int_\Omega p \operatorname{div} \varphi\,dx=-\int_\Omega f\cdot \varphi\,dx+\int_\Omega f^\alpha \cdot D_\alpha \varphi\,dx
$$
holds for any $\varphi\in C^\infty_{\cD}(\Omega)^2$.
Similarly, we say that $(u, p)\in W^{1,q}_{\cD}(\Omega)^2\times L^q(\Omega)$ is a weak solution of the adjoint problem 
$$
\begin{cases}
\operatorname{div} u=g &\text{in }\, \Omega,\\
\cL^* u+\nabla p=f+D_\alpha f^\alpha &\text{in }\, \Omega,\\
\cB^* u+p \nu=\nu_\alpha f^\alpha  & \text{on }\, \cN,\\
u=0 & \text{on }\, \cD,
\end{cases}
$$
if $\operatorname{div} u=g$ a.e. in $\Omega$ and 
$$
\int_\Omega A^{\alpha\beta} D_\beta \varphi\cdot D_\alpha u\,dx+\int_\Omega p \operatorname{div} \varphi\,dx=-\int_\Omega f\cdot \varphi\,dx+\int_\Omega f^\alpha \cdot D_\alpha \varphi\,dx
$$
holds for any $\varphi\in C^\infty_{\cD}(\Omega)^2$.
%For $u\in W^{1,1}(\Omega)^2$ and $g\in L^1(\Omega)^2$, the equation $\operatorname{div} u=g$ in $\Omega$ holds in almost everywhere sense.

We remark that even when $\Omega$ is irregular so that neither the outer normal nor the trace of a function in $W^{1,q}(\Omega)$ on the boundary is defined, the weak formulations above make sense because no boundary terms appear there.

The following is the definition of the Green function for the Stokes system.
Let $I$ be the $2\times 2$ identity matrix and $\delta_y$ is the Dirac delta function concentrated at $y$.

\begin{definition}		\label{D1}
We say that a pair $(G, \Pi)$ is a Green function for $\cL$ in $\Omega$ if it satisfies the following properties.
\begin{enumerate}[$(i)$]
\item
For any $y\in \Omega$, 
$$
G(\cdot, y)\in W^{1,1}_{\cD}(\Omega)^{2\times 2}, \quad \Pi(\cdot,y)\in L^1(\Omega)^{1\times2}.
$$
\item
For any $y\in \Omega$, $(G(\cdot,y), \Pi(\cdot,y))$ satisfies 
\begin{equation}		\label{230205_eq1}
\begin{cases}
\operatorname{div} G(\cdot,y)=0 &\text{in }\, \Omega,\\
\cL G(\cdot,y)+\nabla \Pi(\cdot,y)=-\delta_y I &\text{in }\, \Omega,\\
\cB G(\cdot,y)+\nu\Pi(\cdot, y) =0 & \text{on }\, \cN,\\
u=0 & \text{on }\, \cD,
\end{cases}
\end{equation}
in the sense that we have 
$$
\operatorname{div} G_{\cdot k}(\cdot,y)=0 \, \text{ a.e. in }\, \Omega
$$
and 
$$
\int_\Omega  A^{\alpha\beta}D_\beta G_{\cdot k}(\cdot, y)\cdot D_\alpha \varphi\,dx+\int_\Omega \Pi_k(\cdot,y)\operatorname{div}\varphi\,dx=\varphi_k(y)
$$
for any $k\in \{1,2\}$ and $\varphi\in W^{1, \infty}_{\cD}(\Omega)^2\cap C(\Omega)^2$, where $G_{\cdot k}(\cdot, y)$ is the $k$th column of $G(\cdot, y)$.
\item
If $(u, p)\in W^{1,2}_{\cD}(\Omega)^2\times L^2(\Omega)$ is a weak solution of the adjoint problem
\begin{equation}		\label{230209_eq4}
\begin{cases}
\operatorname{div} u=g &\text{in }\, \Omega,\\
\cL^* u+\nabla p=f+D_\alpha f^\alpha &\text{in }\, \Omega,\\
\cB^* u+p\nu =\nu_\alpha f^\alpha  & \text{on }\, \cN,\\
u=0 & \text{on }\, \cD,
\end{cases}
\end{equation}
where $f, f^\alpha\in L^\infty(\Omega)^2$, and $g\in L^\infty(\Omega)$, then for a.e. $y\in \Omega$, we have 
$$
u(y)=-\int_\Omega G(\cdot, y)^{\top} f\,dx+\int_\Omega D_\alpha G(\cdot, y)^\top f^\alpha \,dx+\int_\Omega \Pi(\cdot,y)^{\top}g\,dx,
$$
where $G(\cdot,y)^{\top}$ and $\Pi(\cdot,y)^{\top}$ are the transposes of $G(\cdot ,y)$ and $\Pi(\cdot,y)$.
\end{enumerate}
The Green function for the adjoint operator $\cL^*$ is defined similarly.
\end{definition}

We remark that, under Assumptions \ref{A1} and \ref{A2} below, the property $(iii)$ in Definition \ref{D1} and the solvability result in  Lemma \ref{230127_lem2} give 
the uniqueness of a Green function in the sense that if $(\tilde{G}, \tilde{\Pi})$ is another Green function satisfying the properties in Definition \ref{D1}, then for any $\phi\in L^\infty(\Omega)^2$ and $\varphi\in L^\infty(\Omega)$, we have 
$$
\int_\Omega \big(G(\cdot, y)^{\top}-\tilde{G}(\cdot, y)^\top\big) \phi\,dx=\int_{\Omega} \big(\Pi(\cdot, y)^{\top}-\tilde{\Pi}(\cdot,y)^{\top}\big)\varphi\,dx=0
$$
for a.e. $y\in \Omega$.

%========================================
\subsection{Main result}	\label{S2_2}
%========================================

%In this subsection, we state our main result concerning the Green function of the mixed problem.
%For this, we impose the following regularity assumptions on $\partial \Omega$ and $\Gamma$.
%In this paper, we work on the so-called Reifenberg flat domains, which is defined below.

Throughout the paper, we assume that $\Omega$ is a Reifenberg flat domain in the following sense.
Let $R_0>0$ be a fixed  constant.

\begin{assumption}		\label{A1}
Let $\gamma\in [0,1/96]$.
For any $x_0\in \partial \Omega$ and $R\in (0, R_0]$, there exists a coordinate system depending on $x_0$ and $R$ such that in this coordinate system (called the coordinate system associated with $(x_0, R)$), we have 
$$
\{y:x_{0}^1+\gamma R<y^1\}\cap B_R(x_0) \subset \Omega_R(x_0)\subset \{y: x_{0}^1-\gamma R<y^1\} \cap B_R(x_0),
$$
where $x_{0}^1$ is the first coordinate of $x_0$ in the coordinate system.
\end{assumption}

We also impose the following assumption on the components $\cD$ and $\cN$ separated by $\Gamma$.

\begin{assumption}		\label{A2}
Let $\kappa\in (0,1)$.

\begin{enumerate}[$(i)$]
\item
There exists $y_0\in \partial \Omega$ such that 
$$
\big(\partial \Omega\cap B_{\kappa R_0}(y_0)\big)\subset \cN.
$$
\item
For any $x_0\in \Gamma$ and $R\in (0, R_0]$, there exists $z_0\in \cD\cap B_{R}(x_0)$ such that 
$$
B_{\kappa R}(z_0)\subset B_R(x_0), \quad \big(\partial \Omega \cap B_{\kappa R}(z_0)\big)\subset  \cD.
$$
\end{enumerate}
\end{assumption}

%\begin{remark}
%A domain $\Omega$ satisfying Assumption \ref{A1} is called Reifenberg flat.
%
%Regarding the flatness parameter $\gamma$ in Assumption \ref{A1}, one can just set $\gamma=1/96$ instead of $\gamma\in [0, 1/96]$ because our main theorem (Theorem \ref{M1}) is independent of the size of $\gamma$ as long as $\gamma\le 1/96$. Clearly,  boundaries satisfying Assumption \ref{A1} also satisfy Assumption \ref{A1} with $\gamma=1/96$.
%See also Remark \ref{230216_rmk3} below.
%\end{remark}

The main result of the paper reads as follows.

\begin{theorem}		\label{M1}
Let $\Omega$ be a bounded Reifenberg flat domain in $\bR^2$  satisfying Assumptions \ref{A1} and \ref{A2}. 
Then there exist Green functions $(G, \Pi)$ and $(G^*, \Pi^*)$ for $\cL$ and $\cL^*$, respectively, such that
\begin{equation}		\label{230209_eq5e}
G(x,y)=G^*(y,x)^\top
\end{equation}
and that  
\begin{equation}		\label{230209_eq5d}
|G(x,y)|\le C\log \bigg(\frac{\operatorname{diam}(\Omega)}{|x-y|}\bigg)+C
\end{equation}
for all $x,y\in \Omega$ with $x\neq y$.
Moreover, the following estimates hold.
\begin{enumerate}[$(i)$]
\item
For any $x,y\in \Omega$ and $R\in (0, R_0]$, we have 
\begin{equation}		\label{230209_eq5}
\|DG(\cdot,y)\|_{L^q(\Omega_R(x))}+\|\Pi(\cdot, y)\|_{L^q(\Omega_R(x))}\le C_q R^{-1+2/q},
\end{equation}
where $1\le q<2$.
\item
There exists $q_0>2$ such that for any $y\in \Omega$ and $R\in (0, R_0]$, we have 
\begin{equation}		\label{230209_eq5a}
\|DG(\cdot,y)\|_{L^{q}(\Omega\setminus \overline{B_R(y)})}+\|\Pi(\cdot,y)\|_{L^{q}(\Omega\setminus \overline{B_R(y)})}\le C_qR^{-1+2/q},
\end{equation}
where $2<q\le q_0$.
Moreover, for any $x\in \overline{\Omega}$ satisfying $|x-y|\ge R$, we have 
\begin{equation}		\label{230209_eq5b}
[G(\cdot,y)]_{C^{1-2/q}(\Omega_{R/16}(x))}\le C_q R^{-1+2/q}.
\end{equation}
\item
For any $y\in \Omega$, we have 
\begin{equation}		\label{230209_eq5c}
\|DG(\cdot,y)\|_{L^{2, \infty}(\Omega)}+\|\Pi(\cdot,y)\|_{L^{2, \infty}(\Omega)}\le C.
\end{equation}
\end{enumerate}
In the above, $(C, q_0)=(C, q_0)(\lambda, R_0, \kappa, \operatorname{diam}(\Omega))$, and $C_q$ depends also on $q$.
\end{theorem}

Related to the theorem above, we have a few remarks.

\begin{remark}		\label{230216_rmk2}
Let $(u, p)\in W^{1,2}_{\cD}(\Omega)^2\times L^2(\Omega)$ be a weak solution of the problem 
$$
\begin{cases}
\operatorname{div} u=0 &\text{in }\, \Omega,\\
\cL u+\nabla p=f+D_\alpha f^\alpha &\text{in }\, \Omega,\\
\cB u+p \nu=\nu_\alpha f^\alpha  & \text{on }\, \cN,\\
u=0 & \text{on }\, \cD,
\end{cases}
$$
where $f, f^\alpha\in L^\infty(\Omega)^2$.
Then by the counterpart of $(iii)$ in Definition \ref{D1} for $(G^*, \Pi^*)$, we have that for a.e. $y\in \Omega$, 
$$
u(y)=-\int_\Omega G^*(\cdot, y)^{\top} f\,dx+\int_\Omega D_\alpha G^*(\cdot, y)^\top f^\alpha \,dx.
$$
This along with \eqref{230209_eq5e} yields that  
$$
u(y)=-\int_\Omega G(y,\cdot )f\,dx+\int_\Omega D_\alpha G(y, \cdot) f^\alpha \,dx.
$$
\end{remark}

\begin{remark}		\label{230216_rmk3}
Note that Lipschitz domains whose Lipschitz constants are not greater than $1/96$ satisfy Assumption \ref{A1}.
Hence, our main theorem (Theorem \ref{M1}) holds for all such Lipschitz domains.
On the other hand, the embeddings and solvability results presented in Section \ref{S3} also hold for Lipschitz domains whenever the Lipschitz constants are bounded.
Thus, in the main theorem, instead of Assumption \ref{A1}, one can assume that the domain is Lipschitz with a bounded Lipschitz constant.
This shows that the flatness of the boundary is irrelevant to Theorem \ref{M1} for a Lipschitz domain.
Moreover, in this case, the conditions in Assumption \ref{A2} can be relaxed to the framework of \cite{MR3320459} with the Ahlfors-David condition.
The relaxation relies on the Sobolev-Poincar\'e type inequality derived in \cite[Appendix]{MR3040944}.
\end{remark}

%In the corollary below, we present an interior pointwise bound  for $(DG, \Pi)$ when the coefficients of $\cL$ are of DMO (Dini mean oscillation), which follows from $C^1$-estimates derived in \cite{MR3912724}.
%One may refer to the proof of \cite[Theorem 3.5]{MR3906316} combined with \cite[Remark 3.6]{MR3906316}, where the authors proved the corresponding pointwise bounds for Green functions of Dirichlet problems.
%Recall that a measurable function $\omega:(0,1]\to [0, \infty)$ is said to be a Dini function if there are positive constants $c_1$ and $c_2$ such that 
%$$
%c_1\omega(t)\le \omega(s)\le c_2\omega(t), \quad 0<\frac{t}{2}\le s\le t\le 1,
%$$
%and  $\omega$ satisfies the Dini condition 
%$$
%\int_0^1 \frac{\omega(t)}{t}\,dt<\infty.
%$$
%
%\begin{corollary}		\label{230216_cor1}
%Let $(G, \Pi)$ be the Green function for $\cL$ constructed in Theorem \ref{M1} under Assumptions \ref{A1} and \ref{A2}.
%Suppose that the coefficients $A^{\alpha\beta}$ of $\cL$ are of DMO in (the interior of) $\Omega$, i.e., there exists a Dini function $\omega:(0, 1]\to [0, \infty)$ such that for any $z\in \Omega$ and $r\in (0,1]$ satisfying $B_{r}(z)\subset \Omega$,  we have 
%$$
%\dashint_{B_r(z)} \bigg|A^{\alpha\beta}(\xi)-\big(A^{\alpha\beta}\big)_{B_r(z)}\bigg|\,d\xi\le \omega(r).
%$$
%Then for any $x,y\in \Omega$ with $0<|x-y|\le \frac{1}{2}\operatorname{dist}(y, \partial \Omega)$, we have 
%$$
%|D_x G(x,y)|+|\Pi(x,y)|\le \frac{C}{|x-y|},
%$$
%where $C=C(\lambda, R_0,\kappa, \operatorname{diam}(\Omega), \omega)$.
%\end{corollary}

%========================================
\section{Auxiliary results}	\label{S3}
%========================================

In this section, we present some auxiliary results that will be used to prove the main theorem.
%Hereafter in the paper, we use the following notation.
%\begin{notation}
%For a given function $f$, if there is a continuous version of $f$, that is, there is a continuous function $\tilde{f}$ such that $\tilde{f}=f$ in the almost everywhere sense, then we replace $f$ with $\tilde{f}$ and denote the version again by $f$.
%\end{notation}

%========================================
\subsection{Embeddings}	\label{S3_1}
%========================================
For $q\in (1,2)$, we denote by $q^\star$ its Sobolev conjugate, that is, 
$q^\star=2q/(2-q)$.
Recall that $\Omega$ is a bounded domain in $\bR^2$.

\begin{lemma}		\label{240213_lem1}
Let $\Omega$ satisfy Assumption \ref{A1}, $q\in (1,2)$, and $u\in W^{1,q}(\Omega)$.
Then  for any $x_0\in \overline{\Omega}$ and $R\in (0, R_0]$, we have 
$$
\|u-(u)_{\Omega_{R/8}(x_0)}\|_{L^{q^\star}(\Omega_{R/8}(x_0))}\le C\|Du\|_{L^q(\Omega_R(x_0))},
$$
where $C=C(q)$.
\end{lemma}

\begin{proof}
See \cite[Lemma 3.1]{CK2023}.
\end{proof}

\begin{lemma}		\label{240213_lem2}
Let $\Omega$ satisfy Assumption \ref{A1}, $s\in (2, \infty]$, and $u\in W^{1,s}(\Omega)$.
Then for any $x_0\in \overline{\Omega}$ and $R\in (0, R_0]$, we have 
\begin{equation}		\label{210804@eq3}
[u]_{C^{1-2/s}(\Omega_{R/8}(x_0))}\le C\|Du\|_{L^s(\Omega_R(x_0))}
\end{equation}
and 
\begin{equation}		\label{230804_eq3a}
\|u\|_{L^\infty(\Omega_{R/8}(x_0))}\le CR^{1-2/s}\|Du\|_{L^q(\Omega_R(x_0))}+CR^{-2}\|u\|_{L^1(\Omega_{R/8}(x_0))},
\end{equation}
where $C=C(s)$.
\end{lemma}

\begin{proof}
See \cite[Lemma 3.3]{CK2023}.
\end{proof}

\begin{lemma}		\label{240213_lem3}
Let $\Omega$ satisfy Assumptions \ref{A1} and \ref{A2} $(i)$, $q\in (1,2)$, and $u\in W^{1,q}_{\cN}(\Omega)$.
Then we have 
$$
\|u\|_{L^{q^\star}(\Omega_{R_0/8}(y_0))}\le C\|Du\|_{L^q(\Omega_{R_0}(y_0))},
$$
where $C=C(\kappa, q)$.
Here, $y_0$ is the point given in Assumption \ref{A2} $(i)$.
\end{lemma}

\begin{proof}
We follow the proof  of  \cite[Corollary 3.2 (a)]{MR4261267} with obvious modifications.
We extend $u$ by zero on $B_{\kappa R_0/8}(y_0)\setminus \Omega$ so that 
$$
u\in W^{1,q}(B_{\kappa R_0/8}(y_0)).
$$
Since $\Omega$ satisfies Assumption \ref{A1}, it is easily seen that 
$$
|B_{\kappa R/8}(y_0)\setminus \Omega|\ge C (\kappa R/8)^2,
$$
where $C$ is a universal constant.
Hence by the boundary version of the Poincar\'e inequality, we have 
\begin{equation}		\label{240213_A1}
\|u\|_{L^{q^*}(\Omega_{\kappa R_0/8}(y_0))}\le C\|Du\|_{L^q(\Omega_{\kappa R_0/8}(y_0))},
\end{equation}
where $C=C(q)$.
Notice that  from the triangle inequality and H\"older's inequality, we get  
$$
\begin{aligned}
\|u\|_{L^{q^*}(\Omega_{R_0/8}(y_0))}&\le \|u-(u)_{\Omega_{R_0/8}(y_0)}\|_{L^{q^*}(\Omega_{R_0/8}(y_0))}\\
&\quad +\|(u)_{\Omega_{R_0/8}(y_0)}-(u)_{\Omega_{\kappa R_0/8}(y_0)}\|_{L^{q^*}(\Omega_{R_0/8}(y_0))}\\
&\quad +\|(u)_{\Omega_{\kappa R_0/8}(y_0)}\|_{L^{q^*}(\Omega_{R_0/8}(y_0))}\\
&\le C \|u-(u)_{\Omega_{R_0/8}(y_0)}\|_{L^{q^*}(\Omega_{R_0/8}(y_0))}+C\|u\|_{L^{q^*}(\Omega_{\kappa R_0/8}(y_0))}.
\end{aligned}
$$
This, Lemma \ref{240213_lem1}, and \eqref{240213_A1} yield the desired inequality.
\end{proof}

\begin{lemma}		\label{240213_lem4}
Let $\Omega$ satisfy Assumptions \ref{A1} and \ref{A2} $(ii)$, $q\in (1,2)$, and $u\in W^{1,q}_{\cD}(\Omega)$.
Then for any $x_0\in \Gamma$ and $R\in (0, R_0]$, we have 
$$
\|u\|_{L^{q^\star}(\Omega_{R/8}(x_0))}\le C\|Du\|_{L^q(\Omega_R(x_0))},
$$
where $C=C(\kappa, q)$.
\end{lemma}

\begin{proof}
The proof is similar to  that of Lemma \ref{240213_lem3}.
See also \cite[Corollary 3.2 (a)]{MR4261267}.
\end{proof}

Combining the local inequalities in Lemmas \ref{240213_lem1}, \ref{240213_lem3}, and \ref{240213_lem4}, we obtain the following Sobolev inequality over $\Omega$.

\begin{lemma}		\label{240212_lem3}
Let $\Omega$ satisfy Assumptions \ref{A1} and \ref{A2}, $q\in (1,2)$, and $u\in W^{1,q}_{\cN}(\Omega)$.
Then we have 
$$
\|u\|_{L^{q^\star}(\Omega)}\le C\|Du\|_{L^q(\Omega)},
$$
where $C=C(R_0, \kappa, \operatorname{diam}(\Omega), q)$.
The same result also holds if $\cN$ is replaced with $\cD$.
\end{lemma}

\begin{proof}
We only prove the case with $\cN$ because the other case is similar.
Due to a covering argument, it suffices to show that for any $x_0\in \overline{\Omega}$, we have
\begin{equation}		\label{240212_eq3}
\|u\|_{L^{q^\star}(\Omega_{R_0/16}(x_0))}\le C\|Du\|_{L^q(\Omega)}.
\end{equation}
Let $y_0$ be the point given in Assumption \ref{A2} $(i)$.
For  $x_0\in \overline{\Omega}$, 
we choose a chain of subdomains $\Omega_{R_0/16}(z_i)$, $i\in \{1,\ldots, m\}$, with $z_i\in \overline{\Omega}$ satisfying
$$
z_1=x_0, \quad z_m=y_0, \quad |z_i-z_{i+1}|\le R_0/16, \quad i\in \{1,\ldots, m-1\},
$$
where $m=m(R_0, \operatorname{diam}(\Omega))$.
From the triangle inequality it follows that  
$$
\|u\|_{L^{q^\star}(\Omega_{R_0/16}(z_1))}\le \sum_{i=1}^{m-1} \|u-(u)_{\Omega_{R_0/8}(z_{i+1})}\|_{L^{q^\star}(\Omega_{R_0/16}(z_{i}))}+\|u\|_{L^{q^\star}(\Omega_{R_0/16}(z_m))}.
$$
Note that by Lemma \ref{240213_lem1},   
$$
\begin{aligned}
\|u-(u)_{\Omega_{R_0/16}(z_{i+1})}\|_{L^{q^\star}(\Omega_{R_0/16}(z_{i}))}
&\le C\|u-(u)_{\Omega_{R_0/8}(z_{i})}\|_{L^{q^\star}(\Omega_{R_0/8}(z_{i}))}\\
&\le C\|Du\|_{L^q(\Omega)}, 
\end{aligned}
$$
and that by Lemma \ref{240213_lem3},
$$
\|u\|_{L^{q^\star}(\Omega_{R_0/16}(z_m))}\le C\|Du\|_{L^q(\Omega)}.
$$
Combining these together, we get  \eqref{240212_eq3}.
\end{proof}

%========================================
\subsection{Solvability}	\label{S3_2}
%========================================

%In this subsection, we provide soem solvability results for divergence equations and  mixed boundary value problems.

In \cite[Lemma 7.2]{MR4261267}, the $W^{1,q}_{\cN}$ solvability of the divergence equation was established under the assumption that  the separation between $\cD$ and $\cN$ is Riefenberg flat.
This assumption can be relaxed to a weaker condition as in Assumption \ref{A2}; see Lemma \ref{230126_lem1} below.

Throughout this subsection, we assume that $\Omega$ is a bounded Reifenberg flat domain in $\bR^2$ satisfying Assumptions \ref{A1} and \ref{A2}.

\begin{lemma}		\label{230126_lem1}
Let $f\in L^q(\Omega)$ with $q\in (1, \infty)$. 
Then there exists $u\in W^{1,q}_{\cN}(\Omega)^2$ such that 
$$
\operatorname{div} u=f \, \text{ in }\, \Omega, \quad  \|Du\|_{L^q(\Omega)}\le C \|f\|_{L^q(\Omega)}, 
$$
where $C=C(R_0, \kappa, \operatorname{diam} (\Omega),q)$.
The same result also holds if $\cN$ is replaced with $\cD$.
\end{lemma}

\begin{proof}
The proof is similar to that of \cite[Lemma 7.2]{MR4261267}.
We only prove the case with $\cN$.
%The proof is similar to that of  \cite[Lemma 7.2]{MR4261267} with obvious modifications.
Due to Assumption \ref{A2} $(ii)$, we can fix $z_0\in \partial \Omega$ such that 
$$
\big(\partial \Omega \cap B_{\kappa R_0}(z_0)\big) \subset \cD.
$$
We take the Whitney decomposition of the open set $\Omega_{\kappa R_0/8}(z_0)$, that is, 
$$
\Omega_{\kappa R_0/8}(z_0)=\cup_i Q_i,
$$
where $Q_i$ are disjoint cubes satisfying 
$$
\operatorname{diam}(Q_i)\le \operatorname{dist}\big(\overline{Q_i},  \big(\Omega_{\kappa R_0/8}(z_0)\big)^c\big)\le 4\operatorname{diam}(Q_i).
$$
We extend $Q_i$ to $\hat{Q}_i$ in the way that 
$$
\hat{Q}_i-x_i=8(Q_i-x_i),
$$
where $x_i$ denotes the center of $Q_i$.
Define 
$$
\hat{\Omega}=\Omega\cup \big(\cup_i \hat{Q}_i\big).
$$
Then it is easily seen that 
$$
\cN\subset \partial \hat{\Omega}, \quad C_0^{-1} (\kappa R_0)^2 \le |\hat{\Omega}\setminus \Omega|\le C(\kappa R_0)^2,
$$
where $C>1$ is a universal constant.
Moreover, by following the proof of \cite[Lemma 7.2]{MR4261267}, we see that  $\hat{\Omega}$ is a John domain.

Now we define the extension of $f$,  denoted by $\hat{f}$, to be 
$$
\hat{f}=
\begin{cases}
f &\text{in }\,  \Omega,\\
\displaystyle-\frac{1}{|\hat{\Omega}\setminus \Omega|} \int_\Omega f\,dx & \text{in }\, \hat{\Omega}\setminus \Omega.
\end{cases}
$$
Then, $(\hat{f})_{\hat{\Omega}}=0$ and 
\begin{equation}		\label{240212_eq4}
\|\hat{f}\|_{L^q(\hat{\Omega})}\le C_1\|f\|_{L^q(\Omega)},
\end{equation}
where $C_1=C_1(R_0, \kappa, \operatorname{diam}(\Omega), q)$.
Since $\hat{\Omega}$ is a John domain, we can apply \cite[Theorem 4.1]{MR4261267} to find  $u\in W^{1,q}_0(\hat{\Omega})^2$ satisfying
$$
\operatorname{div} u=\hat{f} \, \text{ in }\, \hat{\Omega}, \quad \|Du\|_{L^q(\hat{\Omega})}\le C_2 \|\hat{f}\|_{L^q(\hat{\Omega})},
$$
where $C_2=C_2(R_0, \operatorname{diam}(\Omega), q)$.
From this combined with \eqref{240212_eq4} and the fact that $\cN\subset \partial \hat{\Omega}$, we get 
$$
u\in W^{1,q}_{\cN} (\Omega)^2, \quad \operatorname{div} u=f \, \text{ in }\, \Omega, \quad \|Du\|_{L^q(\hat{\Omega})}\le C_1\cdot C_2 \|f\|_{L^q(\Omega)}.
$$
The lemma is proved.
\end{proof}

%The rest of this subsection is devoted to establishing $L^q$-estimates and solvability for the mixed boundary value problems when $q$ is close to $2$.
%The following lemma is about the case with $q=2$.
Since Lemma \ref{230126_lem1} is now available, by following the proof of \cite[Lemma 3.2]{MR3693868}, we obtain the following $L^2$-estimate and solvability for the mixed boundary value problem.
Recall that we do not impose any regularity assumptions on the coefficients $A^{\alpha\beta}$ of the operator $\cL$.

\begin{lemma}		\label{230127_lem2}
Let $f^\alpha\in L^2(\Omega)^2$ and $g\in L^2(\Omega)$.
Then there exists a unique $(u, p)\in W^{1,2}_{\cD}(\Omega)^2\times L^2(\Omega)$ satisfying 
\begin{equation}		\label{230130_eq1}
\begin{cases}
\operatorname{div} u=g &\text{in }\, \Omega,\\
\cL u+\nabla p=D_\alpha f^\alpha &\text{in }\, \Omega,\\
\cB u+p \nu=\nu_\alpha f^\alpha  & \text{on }\, \cN,\\
u=0 & \text{on }\, \cD.
\end{cases}
\end{equation}
Moreover, we have 
$$
\|Du\|_{L^2(\Omega)}+\|p\|_{L^2(\Omega)}\le C \big(\|f^\alpha\|_{L^2(\Omega)}+\|g\|_{L^2(\Omega)}\big),
$$
where $C=C(\lambda,  R_0, \kappa, \operatorname{diam}(\Omega))$.
\end{lemma}

\begin{lemma}		\label{230130_lem1}
Let $(u, p)\in W^{1,2}_{\cD}(\Omega)^2\times L^2(\Omega)$ satisfy  \eqref{230130_eq1} with $f^\alpha\in L^2(\Omega)^2$ and $g\in L^2(\Omega)$.
Then for any $x_0\in \overline{\Omega}$ and $R\in (0, R_0]$,
we have 
$$
\begin{aligned}
(|Du|^2+|p|^2)_{\Omega_{R/16}(x_0)}^{1/2}
&\le \theta (|p|^2)_{\Omega_{R}(x_0)}^{1/2}\\
&\quad+C (|Du|^q+|p|^q)_{\Omega_{R}(x_0)}^{1/q}+C (|f^\alpha|^2+|g|^2)_{\Omega_{R}(x_0)}^{1/2},
\end{aligned}
$$
where $q\in (1,2)$, $\theta\in (0,1)$, and $C=C(\lambda, R_0, \kappa,  \operatorname{diam}(\Omega), q, \theta)$.
\end{lemma}

\begin{proof}
Based on a covering argument, it suffices to consider the following four cases:
$$
\text{(i) }\, B_{R}(x_0)\subset \Omega, \quad \text{(ii) }\, x_0\in \partial \Omega, \quad B_{R}(x_0)\cap \partial \Omega\subset \cN,
$$
$$
\text{(iii) }\, x_0\in \partial \Omega, \quad B_{R}(x_0)\cap \partial \Omega\subset \cD, \quad \text{(iv) }\, x_0\in \Gamma.
$$
For the first two cases, we refer the reader to \cite[Lemma 3.6]{MR3809039}, where the authors obtained the corresponding estimate for solutions to pure conormal derivative problems.
See also \cite[Lemma 3.8]{MR3758532} for the third case, where pure Dirichlet problems were considered.

We present the detailed proof for (iv), in which we need to deal with the mixed Dirichlet-conormal derivative boundary conditions.
Denote 
$B_r=B_r(x_0)$ and $ \Omega_r=\Omega_r(x_0)$ for all $r>0$.
Let $\eta$ be a smooth function in $\bR^2$ such that 
$$
0\le \eta\le1, \quad \eta\equiv 1 \, \text{ in }\, B_{R/16}, \quad \operatorname{supp} \eta \subset B_{R/8}, \quad |\nabla \eta|\le CR^{-1}.
$$
By applying $\eta^2 u$ as a test function to \eqref{230130_eq1} and then, using both H\"older's and Young's inequalities, we obtain for $\theta\in (0,1)$ that 
$$
\begin{aligned}
\int_{\Omega} \eta^2 |Du|^2\,dx
&\le \theta\int_{\Omega_{R/8}} |p|^2\,dx\\
&\quad + \frac{C}{R^2}\int_{\Omega_{R/8}} |u|^2\,dx+C\int_{\Omega_{R/8}} \big(|f^\alpha|^2+|g|^2\big)\,dx,
\end{aligned}
$$
where  by H\"older's inequality and Lemma \ref{240213_lem4},  
\begin{equation}		\label{230227_eq1}
\frac{1}{R^2}\int_{\Omega_{R/8}} |u|^2\,dx\le  C R^{2-4/q}\|u\|^2_{L^{q^*}(\Omega_{R/8})} \le CR^{2-4/q}\|Du\|_{L^q(\Omega_R)}^2.
\end{equation}
Hence we get 
$$
\begin{aligned}
\int_{\Omega} \eta^2 |Du|^2\,dx
&\le \theta\int_{\Omega_{R}} |p|^2\,dx\\
&\quad +CR^{2-4/q}\bigg(\int_{\Omega_{R}} |Du|^q\,dx\bigg)^{2/q}+C\int_{\Omega_{R}} \big(|f^\alpha|^2+|g|^2\big)\,dx,
\end{aligned}
$$
where $C=C(\lambda, \kappa, q, \theta)$. 
This gives the desired estimate for $Du$.

Next, to get the estimate for $p$,  we apply Lemma \ref{230126_lem1} to find $v\in W^{1,2}_{\cD}(\Omega)^2$ satisfying that 
$$
\operatorname{div} v=p \chi_{\Omega_{R/16}} \, \text{ in }\, \Omega
$$
and
\begin{equation}		\label{230130_eq3}
\|Dv\|_{L^2(\Omega)}\le C \|p\chi_{\Omega_{R/16}}\|_{L^2(\Omega)}=C \|p\|_{L^2(\Omega_{R/16})},
\end{equation}
where $C=C(R_0, \kappa, \operatorname{diam}(\Omega))$.
By testing \eqref{230130_eq1} with $\eta^2 v$,  we obtain
\begin{equation}		\label{230227_eq2}
\begin{aligned}
&\int_{\Omega} \eta^2 |p|^2\,dx\\
&\le C_\varepsilon \int_{\Omega} \eta^2|Du|^2\,dx+C_\varepsilon R^{2-4/q} \bigg(\int_{\Omega_R} |p|^q\,dx\bigg)^{2/q}+C_\varepsilon\int_{\Omega_{R}} |f^\alpha|^2\,dx \\
&\quad +\frac{\varepsilon}{R^{4-4/q}}\bigg(\int_{\Omega_{R/8}}  |v|^{q/(q-1)}\,dx\bigg)^{2(q-1)/q}+\varepsilon\int_{\Omega_{R}} |Dv|^2\,dx,
\end{aligned}
\end{equation}
where $\varepsilon\in (0,1)$ and $C_\varepsilon=C_{\varepsilon}(\lambda, \varepsilon,q)$.
By Lemma \ref{240213_lem4} with $q/(q-1)$ in place of $q^*$, we have 
$$
\begin{aligned}
\frac{\varepsilon}{R^{4-4/q}}\bigg(\int_{\Omega_{R/8}}  |v|^{q/(q-1)}\,dx\bigg)^{2(q-1)/q}
&\le \frac{C\varepsilon}{R^{4-4/q}}\|Dv\|_{L^{2q/(3q-2)}(\Omega_R)}^2\\
&\le C \varepsilon \|Dv\|_{L^2(\Omega_R)}^2.
\end{aligned}
$$
Hence by \eqref{230130_eq3}, we can absorb the last two terms on the right-hand side of \eqref{230227_eq2} to get 
$$
\int_{\Omega} \eta^2 |p|^2\,dx
\le C \int_{\Omega} \eta^2|Du|^2\,dx+C R^{2-4/q} \bigg(\int_{\Omega_R} |p|^q\,dx\bigg)^{2/q} +C\int_{\Omega_{R}} |f^\alpha|^2\,dx.
$$
This implies the desired estimate for $p$.
The lemma is proved.
\end{proof}

By Lemma \ref{230130_lem1} and Gehring's lemma, we get the  following reverse H\"older's inequality.

\begin{lemma}		\label{230130_lem2}
There exists $q_0=q_0(\lambda, R_0, \kappa, \operatorname{diam}(\Omega))>2$ such that if 
$(u, p)\in W^{1,2}_{\cD}(\Omega)^2\times L^2(\Omega)$ satisfies  \eqref{230130_eq1} with   $f^\alpha\in L^\infty(\Omega)^2$ and $g\in L^\infty(\Omega)$,
then for any $q\in [2, q_0]$, $x_0\in \bR^2$, and $R\in (0, R_0]$, we have 
$$
\begin{aligned}
(|Du|^q+|p|^q)_{\Omega_{R/2}(x_0)}^{1/q} \le C (|Du|^2+|p|^2)_{\Omega_{R}(x_0)}^{1/2}+C (|f^\alpha|^q+|g|^q)_{\Omega_{R}(x_0)}^{1/q},
\end{aligned}
$$
where $C=C(\lambda, R_0, \kappa, \operatorname{diam}(\Omega),q)$.
\end{lemma}		

The proposition below is about the $L^q$-estimate and solvability of the mixed problem when $q$ is close to $2$, which follows from Lemma \ref{230130_lem2} with duality; see \cite[Lemma 4.4]{MR3906316}. 

%The lemma is about the solvability of the mixed problem  in $W^{1,q}_{\cD}(\Omega)^2\times L^q(\Omega)$ with $q$ that is close to $2$.
\begin{proposition}		\label{230130_prop1}
Let $q\in [q_0', q_0]$, where $q_0=q_0(\lambda, R_0, \kappa, \operatorname{diam}(\Omega))>2$ is the constant from Lemma \ref{230130_lem2} and $q_0'$ is the conjugate exponent of $q_0$.
Then for any $f^\alpha\in L^q(\Omega)^2$ and $g\in L^q(\Omega)$, there exist a unique $(u, p)\in W^{1,q}_{\cD}(\Omega)^2\times L^q(\Omega)$ satisfying \eqref{230130_eq1}.
Moreover, we have 
$$
\|Du\|_{L^q(\Omega)}+\|p\|_{L^q(\Omega)}\le C \|f^\alpha\|_{L^q(\Omega)}+C\|g\|_{L^q(\Omega)},
$$
where $C=C(\lambda, R_0, \kappa, \operatorname{diam}(\Omega), q)$.
\end{proposition}

When $q\in (2, q_0]$, we can extend the result in Proposition \ref{230130_prop1} as follows.

\begin{corollary}		\label{230131_cor1}
Let $q\in (2, q_0]$, where $q_0=q_0(\lambda, R_0, \kappa, \operatorname{diam}(\Omega))>2$ is the constant from Lemma \ref{230130_lem2}.
Then for any $(u, p)\in W^{1,q}_{\cD}(\Omega)^q\times L^q(\Omega)$ satisfying
\begin{equation}		\label{230131_eq3}
\begin{cases}
\operatorname{div} u=g &\text{in }\, \Omega,\\
\cL u+\nabla p=f+D_\alpha f^\alpha &\text{in }\, \Omega,\\
\cB u+p \nu=\nu_\alpha f^\alpha  & \text{on }\, \cN,\\
u=0 & \text{on }\, \cD,
\end{cases}
\end{equation}
where $f\in L^{2q/(2+q)}(\Omega)^2$, $f^\alpha\in L^q(\Omega)^2$, and $g\in L^q(\Omega)$, we have 
\begin{equation}		\label{230131_eq2}
\|Du\|_{L^q(\Omega)}+\|p\|_{L^q(\Omega)}\le C\|f\|_{L^{2q/(2+q)}(\Omega)}+C \|f^\alpha\|_{L^q(\Omega)}+C\|g\|_{L^q(\Omega)},
\end{equation}
where $C=C(\lambda, R_0, \kappa,  \operatorname{diam}(\Omega), q)$.
Moreover, for any $f\in L^{2q/(2+q)}(\Omega)^2$, $f^\alpha\in L^q(\Omega)^2$, and $g\in L^q(\Omega)$, there exist a unique $(u, p)\in W^{1,q}_{\cD}(\Omega)^2\times L^q(\Omega)$ satisfying \eqref{230131_eq3}.
\end{corollary}

\begin{proof}
By Lemma \ref{230126_lem1},  there exists  
$F^\alpha\in W^{1,2q/(2+q)}_{\cN}(\Omega)^2$ such that 
$$
D_\alpha F^\alpha =f \,\text{ in }\, \Omega
$$
and  
\begin{equation}		\label{230131_eq1}
\|F^\alpha\|_{L^q(\Omega)}\le C\|DF^\alpha\|_{L^{2q/(2+q)}(\Omega)}\le C\|f\|_{L^{2q/(2+q)}(\Omega)},
\end{equation}
where the first inequality is due to Lemma \ref{240212_lem3}.
Note that the problem \eqref{230131_eq3} is equivalent to 
$$
\begin{cases}
\operatorname{div} u=g &\text{in }\, \Omega,\\
\cL u+\nabla p=D_\alpha (f^\alpha+F^\alpha) &\text{in }\, \Omega,\\
\cB u+p \nu=\nu_\alpha (f^\alpha+F^\alpha) & \text{on }\, \cN,\\
u=0 & \text{on }\, \cD.
\end{cases}
$$
Hence, by applying  Proposition \ref{230130_prop1} to the above problem and using \eqref{230131_eq1}, we get the desired result.
\end{proof}

%========================================
\section{Proof of Theorem \ref{M1}}	\label{S4}
%========================================

We adapt the arguments in the proof of \cite[Theorem 2.4]{CK2023}.
Set 
$q_0'=\frac{q_0}{q_0-1}$,
where $q_0=q_0(\lambda, R_0, \kappa, \operatorname{diam}(\Omega))>2$ is the constant from Lemma \ref{230130_lem2}.
We fix a smooth function $\Phi$ defined in $\bR^2$ such that 
$$
0\le \Phi\le 1, \quad \operatorname{supp}\Phi\subset B_1(0), \quad \int_{\bR^2}\Phi\,dx=1.
$$
Let $y\in \Omega$ and $\varepsilon\in (0, R_0]$, and denote
$$
\Phi^{\varepsilon, y}(x)=\varepsilon^{-2} \Phi((x-y)/\varepsilon).
$$
By Corollary \ref{230131_cor1}, for each $k\in \{1,2\}$, there exists a unique 
$$
(v, \pi)=(v^{\varepsilon, y, k}, \pi^{\varepsilon, y, k})\in W^{1,q_0}_{\cD}(\Omega)^{2}\times L^{q_0}(\Omega)
$$
satisfying 
\begin{equation}		\label{230206_eq1}
\begin{cases}
\operatorname{div} v=0 &\text{in }\, \Omega,\\
\cL v+\nabla \pi=-\Phi^{\varepsilon, y} e_k &\text{in }\, \Omega,\\
\cB v+ \pi \nu=0 & \text{on }\, \cN,\\
v=0 & \text{on }\, \cD,
\end{cases}
\end{equation}
where $e_k$ is the $k$th unit vector in $\bR^2$.
Moreover, for any $q\in (2, q_0]$, 
\begin{equation}		\label{230206_eq2}
\|Dv\|_{L^{q}(\Omega)}+\|\pi\|_{L^{q}(\Omega)}\le C\|\Phi^{\varepsilon,y}\|_{L^{2q/(2+q)}(\Omega)}\le C\varepsilon^{-1+2/q},
\end{equation}
where $C=C(\lambda, R_0, \kappa, \operatorname{diam}(\Omega),q)$.
By the Morrey-Sobolev embedding, there is a version $\tilde{v}$ of $v$ such that $\tilde{v}=v$ a.e. in $\Omega$ and $\tilde{v}$ is continuous in $\Omega$.
We define {\em the approximated Green function} 
$$
(G^\varepsilon(\cdot,y), \Pi^\varepsilon(\cdot,y))\in W^{1,q_0}_{\cD}(\Omega)^{2\times 2}\times L^{q_0}(\Omega)^{1\times 2}
$$
{\em for $\cL$} by 
$$
G^{\varepsilon}_{jk}(\cdot,y)=\tilde{v}_j=\tilde{v}^{\varepsilon, y, k}_{j} \, \text{ and }\, \Pi^\varepsilon_k(\cdot,y)=\pi=\pi^{\varepsilon, y, k}.
$$
%Here, $G^\varepsilon(\cdot,y)$ is a $2\times 2$ matrix-valued function and $\Pi^\varepsilon(\cdot,y)$ is a $1\times 2$ vector-valued function.
Notice that for each $k\in \{1,2\}$ and $\varphi\in W^{1, q_0'}_{\cD}(\Omega)^2$, we have 
\begin{equation}		\label{230209_eq3a}
\int_\Omega A^{\alpha\beta} D_\beta G^{\varepsilon}_{\cdot k}(\cdot,y)\cdot D_\alpha \varphi\,dx+\int_\Omega \Pi^\varepsilon(\cdot,y)\operatorname{div}\varphi\,dx=\int_\Omega  \Phi^{\varepsilon, y} \varphi_k\,dx,
\end{equation}
where $G^{\varepsilon}_{\cdot k}(\cdot,y)$ is the $k$th column of $G^\varepsilon(\cdot,y)$.

To prove Theorem \ref{M1}, we will use the following lemma related to uniform estimates for $DG^\varepsilon(\cdot, y)$ and $\Pi^\varepsilon(\cdot, y)$.
%The first lemma is about their local $L^s$-estimate with $s<2$.

\begin{lemma}		\label{230209_lem1}
Let $y\in \Omega$ and $R\in (0, R_0]$.
\begin{enumerate}[$(i)$]
\item
Let $1\le s<2$.
Then for any $x\in \Omega$ and $\varepsilon\in (0, R/8]$, we have 
$$
\|DG^\varepsilon(\cdot,y)\|_{L^{s}(\Omega_{R}(x))}+\|\Pi^\varepsilon(\cdot,y)\|_{L^{s}(\Omega_{R}(x))}\le C_sR^{-1+2/s}.
$$
\item
Let $2<s\le q_0$.
Then for any $\varepsilon\in (0, R_0]$, we have 
$$
\|DG^\varepsilon(\cdot,y)\|_{L^{s}(\Omega\setminus \overline{B_R(y)})}+\|\Pi^\varepsilon(\cdot,y)\|_{L^{s}(\Omega\setminus \overline{B_R(y)})}\le C_sR^{-1+2/s}.
$$
\item
For any $\varepsilon\in (0, R_0]$, we have 
$$
\|DG^\varepsilon(\cdot,y)\|_{L^{2, \infty}(\Omega)}+\|\Pi^\varepsilon(\cdot, y)\|_{L^{2, \infty}(\Omega)}\le C.
$$
\end{enumerate}
In the above, $C=C(\lambda, R_0, \kappa, \operatorname{diam}(\Omega))$ and $C_s$ depends also on $s$.
\end{lemma}

\begin{proof}
In this proof, we denote  
$$
(v, \pi)=(G^\varepsilon_{\cdot k}(\cdot, y), \Pi^\varepsilon_k(\cdot, y)).
$$
Let $f^\alpha\in L^\infty(\Omega)^2$ and $g\in L^\infty(\Omega)$.
By Proposition \ref{230130_prop1}, there exists $(u, p)\in W^{1,q_0}(\Omega)^2\times L^{q_0}(\Omega)$ satisfying 
\begin{equation}		\label{230207_eq6}
\begin{cases}
\operatorname{div} u=g &\text{in }\, \Omega,\\
\cL^* u+\nabla p=D_\alpha f^\alpha &\text{in }\, \Omega,\\
\cB^* u+p \nu=\nu_\alpha f^\alpha  & \text{on }\, \cN,\\
u=0 & \text{on }\, \cD.
\end{cases}
\end{equation}
Moreover, for any $q\in [q_0', q_0]$,  we have 
\begin{equation}		\label{230207@eq1}
\|Du\|_{L^{q}(\Omega)}+\|p\|_{L^{q}(\Omega)}\le C\big(\|f^\alpha\|_{L^{q}(\Omega)}+\|g\|_{L^{q}(\Omega)}\big),
\end{equation}
where $C=C(\lambda, R_0, \kappa, \operatorname{diam}(\Omega),q)$.
By applying $u$ and $v$ as test functions to \eqref{230206_eq1} and \eqref{230207_eq6}, respectively, we obtain that 
\begin{equation}		\label{230209_eq1}
\int_\Omega D_\alpha v\cdot f^\alpha\,dz+\int_\Omega \pi g\,dz=\int_\Omega \Phi^{\varepsilon, y} u_k\,dz.
\end{equation}

Now, we prove the assertion $(i)$.
Thanks to H\"older's inequality, it suffices to consider the case when $q_0'\le s<2$.
Suppose that $f^\alpha$ and $g$ have compact supports in $\Omega_R(x)$.
Since $\varepsilon\le R/8$, \eqref{230209_eq1} implies  
$$
\bigg|\int_{\Omega_{R}(x)} D_\alpha v\cdot f^\alpha\,dz+\int_{\Omega_{R}(x)} \pi g\,dz\bigg|\le C\|u\|_{L^\infty(\Omega_{R/8}(y))}.
$$
Then by \eqref{230804_eq3a} with $q=s':=s/(s-1)>2$, H\"older's inequality, and Lemma \ref{240212_lem3}, we obtain 
\begin{align}
\nonumber
&\bigg|\int_{\Omega_{R}(x)} D_\alpha v\cdot f^\alpha\,dz+\int_{\Omega_{R}(x)} \pi g\,dz\bigg|\\
\nonumber
&\le C R^{1-2/s'}\|Du\|_{L^{s'}(\Omega_R(y))}+CR^{1-2/s}\|u\|_{L^{2s/(2-s)}(\Omega_{R/8}(y))}\\
\label{230209_eq1a}
&\le CR^{1-2/s'}\|Du\|_{L^{s'}(\Omega)}+CR^{1-2/s}\|Du\|_{L^{s}(\Omega)}.
\end{align}
Note that \eqref{230207@eq1} holds for both $q=s$ and $q=s'$ since $q_0'\le s< s'\le q_0$.
Thus,  
$$
\begin{aligned}
&\bigg|\int_{\Omega_{R}(x)} D_\alpha v\cdot f^\alpha\,dz+\int_{\Omega_{R}(x)} \pi g\,dz\bigg|\\
&\le CR^{1-2/s'}\big(\|f^\alpha\|_{L^{s'}(\Omega_{R}(x))}+\|g\|_{L^{s'}(\Omega_{R}(x))}\big)\\
&\quad +CR^{1-2/s}\Big(\|f^\alpha\|_{L^{s}(\Omega_{R}(x))}+\|g\|_{L^{s}(\Omega_{R}(x))}\Big)\\
%&\bigg|\int_{\Omega_{R}(x)} D_\alpha v\cdot f^\alpha\,dz+\int_{\Omega_{R}(x)} \pi g\,dz\bigg|\\
& \le CR^{1-2/s'}\big(\|f^\alpha\|_{L^{s'}(\Omega_{R}(x))}+\|g\|_{L^{s'}(\Omega_{R}(x))}\big).
\end{aligned}
$$
%Since the above inequality holds for all $f^\alpha\in L^\infty(\Omega)^2$ and $g\in L^\infty(\Omega)$ having compact supports in $\Omega_R(x)$, 
By the duality we obtain the desired estimate.

We next prove the assertion $(ii)$.
Thanks to \eqref{230206_eq2}, it suffices to consider the case of $\varepsilon\le R/16$.
Assume that $f^\alpha$ and $g$ have compact supports in $\Omega\setminus \overline{B_R(y)}$.
We consider the following two cases:
$$
s\le 4, \quad s>4.
$$
\begin{enumerate}[i.]
\item
$s\le 4$.
In this case, it holds that 
$$
\frac{2s}{2+s}\le s', \quad s\le \frac{2s'}{2-s'},
$$
where $s':=s/(s-1)$.
Similarly as in \eqref{230209_eq1a}, by \eqref{230804_eq3a} with $q=s>2$, we have that (using $\varepsilon\le R/16$)
\begin{align}
\nonumber
&\bigg|\int_{\Omega\setminus \overline{B_{R}(y)}} D_\alpha v\cdot f^\alpha\,dz+\int_{\Omega\setminus \overline{B_{R}(y)}} \pi g\,dz\bigg|\\
\nonumber
&\le \|u\|_{L^\infty(\Omega_{R/16}(y))}\\
\nonumber
&\le CR^{1-2/s}\|Du\|_{L^s(\Omega_{R/2}(y))}+CR^{1-2/s'}\|u\|_{L^{2s'/(2-s')}(\Omega)}\\
\label{230208_eq1}
&\le CR^{1-2/s}\|Du\|_{L^s(\Omega_{R/2}(y))}+CR^{1-2/s'}\|Du\|_{L^{s'}(\Omega)}.
\end{align}
Let $\eta$ be an infinitely differentiable function in $\bR^2$ such that 
$$
0\le \eta \le 1, \quad \eta\equiv 1 \, \text{ in }\, B_{R/2}(y), \quad \operatorname{supp}\eta \subset B_R(y), \quad |\nabla \eta|\le CR^{-1}.
$$
Since $f^\alpha\equiv g\equiv 0$ in $\Omega_R(y)$, 
$(\eta u, \eta p)$ satisfies 
\begin{equation}		\label{230207@eq3c}
\begin{cases}
\operatorname{div} (\eta u)=H &\text{in }\, \Omega,\\
\cL^* (\eta u)+\nabla (\eta p)=F+D_\alpha F^\alpha &\text{in }\, \Omega,\\
\cB^* (\eta u)+(\eta p) \nu=\nu_\alpha F^\alpha  & \text{on }\, \cN,\\
\eta u=0 & \text{on }\, \cD,
\end{cases}
\end{equation}
where 
$$
F=D_\alpha \eta (A^{\beta \alpha})^\top D_\beta u +p \nabla \eta, \quad F^\alpha=D_\beta \eta(A^{\beta\alpha})^\top u , \quad H=\nabla \eta \cdot u.
$$
From \eqref{230131_eq2} with $q=s$ applied to \eqref{230207@eq3c}, H\"older's inequality, and Lemma \ref{240212_lem3}, it follows that 
\begin{align}
\nonumber
&\|Du\|_{L^{s}(\Omega_{R/2}(y))}+\|p\|_{L^s(\Omega_{R/2}(y))}\\
\nonumber
&\le C  \|F\|_{L^{2s/(2+s)}(\Omega)}+C\|F^\alpha\|_{L^s(\Omega)}+C\|H\|_{L^s(\Omega)}\\
\nonumber
&\le C R^{-1}\big(\|Du\|_{L^{2s/(2+s)}(\Omega_R(y))}+\|p\|_{L^{2s/(2+s)}(\Omega_R(y))}\big)+C R^{-1} \|u\|_{L^s(\Omega_R(y))}\\
\nonumber
&\le C R^{2/s-2/s'} \big(\|Du\|_{L^{s'}(\Omega)}+\|p\|_{L^{s'}(\Omega)}\big)+CR^{2/s-2/s'}\|u\|_{L^{2s'/(2-s')}(\Omega)}\\
\label{230208_eq2a}
&\le C R^{2/s-2/s'} \big(\|Du\|_{L^{s'}(\Omega)}+\|p\|_{L^{s'}(\Omega)}\big).
\end{align}
Combining this together with \eqref{230208_eq1},   we obtain
$$
\begin{aligned}
&\bigg|\int_{\Omega\setminus \overline{B_{R}(y)}} D_\alpha v\cdot f^\alpha\,dz+\int_{\Omega\setminus \overline{B_{R}(y)}} \pi g\,dz\bigg|\\
&\le CR^{1-2/s'}\big(\|Du\|_{L^{s'}(\Omega)}+\|p\|_{L^{s'}(\Omega)}\big),
\end{aligned}
$$
and thus by using \eqref{230207@eq1} with $q=s'$ and the fact that $1-2/s'=-1+2/s$, we see that 
\begin{equation}		\label{230208_eq2b}
\begin{aligned}
&\bigg|\int_{\Omega\setminus \overline{B_{R}(y)}} D_\alpha v\cdot f^\alpha\,dz+\int_{\Omega\setminus \overline{B_{R}(y)}} \pi g\,dz\bigg|\\
&\le CR^{-1+2/s} \big(\|f^\alpha\|_{L^{s'}(\Omega\setminus \overline{B_R(y)})}+\|g\|_{L^{s'}(\Omega\setminus \overline{B_R(y)})}\big).
\end{aligned}
\end{equation}
%Since the above inequality holds for all $f^\alpha\in L^\infty(\Omega)^2$ and $g\in L^\infty(\Omega)$ having compact supports in $\Omega\setminus \overline{B_R(y)}$, 
By the duality we get the desired estimate.
\item
$s>4$.
In this case, we have 
$$
s'<\frac{2s}{2+s}, \quad \frac{2s'}{2-s'}<s.
$$
Similar to \eqref{230208_eq1}, using \eqref{230804_eq3a} with $q=2s'/(2-s')>2$, we obtain that  
$$
\begin{aligned}
&\bigg|\int_{\Omega\setminus \overline{B_{R}(y)}} D_\alpha v\cdot f^\alpha\,dz+\int_{\Omega\setminus \overline{B_{R}(y)}} \pi g\,dz\bigg|\\
&\le CR^{2-2/s'}\|Du\|_{L^{2s'/(2-s')}(\Omega_{R/2}(y))}+CR^{1-2/s'}\|Du\|_{L^{s'}(\Omega)}.
\end{aligned}
$$
By the same calculation used in deriving \eqref{230208_eq2a}, we have 
$$
\|Du\|_{L^{2s'/(2-s')}(\Omega_{R/2}(y))}\le  CR^{-1}\big(\|Du\|_{L^{s'}(\Omega)}+\|p\|_{L^{s'}(\Omega)}\big).
$$
Combining these together and utilizing \eqref{230207@eq1} with $q=s'$, we derive \eqref{230208_eq2b}, which implies the desired estimate.
\end{enumerate}
The assertion $(ii)$ is proved.

To prove the assertion $(iii)$, let
$$
A_t=\{x\in \Omega:|DG^\varepsilon(\cdot, y)>t\}, \quad t>0.
$$
If $t\le R_0^{-1}$, then 
\begin{equation}		\label{230209_eq2}
t|A_t|^{1/2}\le R_0^{-1} |\Omega|^{1/2}\le C,
\end{equation}
where $C=C(R_0, \operatorname{diam}(\Omega))$.
If $t>R_0^{-1}$, then by the estimate in the assertion $(ii)$ with $R=t^{-1}<R_0$ and $s=q_0$, we have 
$$
\big|A_t\setminus \overline{B_R(y)}\big|\le \frac{1}{t^{q_0}} \int_{A_t\setminus \overline{B_R(y)}} |DG^\varepsilon(\cdot, y)|^{q_0}\,dz\le Ct^{-2}, 
$$
where $C=C(\lambda, R_0, \kappa, \operatorname{diam}(\Omega), q_0)=C(\lambda, R_0, \kappa, \operatorname{diam}(\Omega))$.
Since 
$$
\big|A_t\cap \overline{B_R(y)}\big|\le CR^2=Ct^{-2},
$$
we get 
$$
t|A_t|^{1/2}\le C.
$$
Combining this together with \eqref{230209_eq2}, we obtain 
$$
\|DG^\varepsilon(\cdot, y)\|_{L^{2, \infty}(\Omega)}\le C.
$$
Similarly, we have the estimate for $\Pi^\varepsilon(\cdot, y)$.
The assertion $(iii)$ is proved.
\end{proof}

We are ready to prove Theorem \ref{M1}.

\begin{proof}[Proof of Theorem \ref{M1}]
Let us fix $s\in (1,2)$.
From Lemma \ref{230209_lem1} $(iii)$ combined with Lemma \ref{240212_lem3},  we have that for $y\in \Omega$ and $\varepsilon\in (0, R_0]$, 
$$
\|G^\varepsilon(\cdot, y)\|_{W^{1,s}_{\cD}(\Omega)}+\|\Pi^\varepsilon(\cdot, y)\|_{L^s(\Omega)}\le C,
$$
where $C=C(\lambda, R_0, \kappa, \operatorname{diam}(\Omega),s)$.
Hence by the weak compactness theorem, there exist a pair 
$$
(G(\cdot, y), \Pi(\cdot, y))\in W^{1,s}_{\cD}(\Omega)^{2\times 2}\times L^s(\Omega)^{1\times 2}
$$
and a sequence $\{\varepsilon_{\rho}\}_{\rho=1}^\infty$ tending to zero such that 
\begin{equation}		\label{230209_eq3}
G^{\varepsilon_\rho}(\cdot, y) \rightharpoonup G(\cdot, y) \, \text{ weakly in }\, W^{1, s}_{\cD}(\Omega)^{2\times 2}
\end{equation}
and
$$
\Pi^{\varepsilon_\rho}(\cdot, y) \rightharpoonup \Pi(\cdot, y) \, \text{ weakly in }\, L^{s}_{\cD}(\Omega)^{1\times 2}.
$$
Moreover, 
\begin{equation}		\label{230209_eq3b}
\|G(\cdot, y)\|_{W^{1,s}_{\cD}(\Omega)}+\|\Pi(\cdot, y)\|_{L^s(\Omega)}\le C.
\end{equation}

Now, we shall show that $(G, \Pi)$ satisfies the properties $(i)$--$(iii)$ in Definition \ref{D1} so that it is a unique Green function for $\cL$ in $\Omega$.
The property $(i)$ follows immediately from \eqref{230209_eq3b}.
By taking $\rho\to \infty$ in \eqref{230209_eq3a} with $\varepsilon_\rho$ in place of $\varepsilon$, we verify the property $(ii)$.
To see the property $(iii)$, let $(u, p)$ be a weak solution of the adjoint problem \eqref{230209_eq4}.
Testing \eqref{230209_eq4} and \eqref{230206_eq1} with $G^{\varepsilon_\rho}_{\cdot k}(\cdot, y)$ and $u$, respectively, 
we have 
$$
\int_{\Omega} \Phi^{\varepsilon, y}u_k\,dz=-\int_{\Omega} G^{\varepsilon_\rho}_{\cdot k} (\cdot, y)\cdot f\,dz+\int_\Omega D_\alpha G_{\cdot k}^{\varepsilon_\rho}(\cdot, y)\cdot f^\alpha\,dz+\int_\Omega \Pi_k^{\varepsilon_\rho}(\cdot, y)g\,dz,
$$
This shows that 
$$
u_k(y)=-\int_\Omega G_{\cdot k}(\cdot, y)\cdot f\,dz+\int_\Omega D_\alpha G_{\cdot k}(\cdot, y)\cdot f^\alpha\,dz+\int_\Omega \Pi_k(\cdot, y)g\,dz,
$$
provided that $y$ is in the Lebesgue set of $u_k$.
Hence, $(G, \Pi)$ satisfies the property $(ii)$, and thus it is the Green function for $\cL$ in $\Omega$.

We next prove the estimates \eqref{230209_eq5}--\eqref{230209_eq5c}.
Let $x,y\in \Omega$, $R\in (0, R_0]$, and $\varphi\in L^\infty(\Omega_R(x))$.
By Lemma \ref{230209_lem1} $(i)$ combined with \eqref{230209_eq3}, we have 
$$
\begin{aligned}
\bigg|\int_{\Omega_R(x)}DG(\cdot, y)\varphi\,dz\bigg|
&=\lim_{\rho\to \infty} \bigg|\int_{\Omega_R(x)} DG^{\varepsilon_\rho}(\cdot, y)\varphi\,dz\bigg|\\
&\le \limsup_{\rho\to \infty} \|DG^{\varepsilon_\rho}(\cdot, y)\|_{L^q(\Omega_R(x))}\|\varphi\|_{L^{q'}(\Omega_R(x))}\\
&\le C_q R^{-1+2/q}\|\varphi\|_{L^{q'}(\Omega_R(x))},
\end{aligned}
$$
where $1\le q<2$, $q'=q/(q-1)$, and $C_q=C_q(\lambda, R_0, \kappa, \operatorname{diam}(\Omega), q)$.
Hence by the duality, 
$$
\|DG(\cdot, y)\|_{L^q(\Omega_R(x))}\le C_q R^{-1+2/q}.
$$
Similarly we have 
$$
\|\Pi(\cdot, y)\|_{L^q(\Omega_R(x))}\le C_q R^{-1+2/q},
$$
and thus, the estimate \eqref{230209_eq5} holds.
By the same reasoning, Lemma \ref{230209_lem1} $(ii)$ yields  \eqref{230209_eq5a}, from which together with \eqref{210804@eq3} we verify \eqref{230209_eq5b}.
We also have \eqref{230209_eq5c} as in the proof of Lemma \ref{230209_lem1} $(iii)$.
Note that by \eqref{230209_eq5b}, 
\begin{equation}		\label{210802@eq3}
\bigg|G(x_0, y)-\dashint_{\Omega_{R/16}(x)} G(z, y)\,dz\bigg|\le C_0
\end{equation}
for all $x_0, x,y\in \Omega$ and $R\in (0, R_0]$ satisfying $|x-y|\ge R$ and $x_0\in \Omega_{R/16}(x)$, where $C_0=C_0(\lambda, R_0, \kappa, \operatorname{diam}(\Omega))$.

To prove \eqref{230209_eq5d}, let $x,y\in \Omega$ with $x\neq y$, and set $r=|x-y|$.
If $r\ge R_0/8$, then since 
\begin{equation}		\label{230209_eq6}
\|G(\cdot, y)\|_{L^1(\Omega)}\le C(\lambda, R_0, \kappa, \operatorname{diam}(\Omega)), 
\end{equation}
by \eqref{210802@eq3} with $x_0=x$ and $R=R_0/8$, we get 
$$
\begin{aligned}
|G(x,y)|
&\le \bigg|G(x, y)-\dashint_{\Omega_{R_0/128}(x)} G(z,y)\,dz\bigg|+\bigg|\dashint_{\Omega_{R_0/128}(x)} G(z,y)\,dz\bigg|\\
&\le C,
\end{aligned}
$$
which gives \eqref{230209_eq5d}.
Otherwise, i.e., if $r<R_0/8$, then we take a point $y_0\in \Omega$ and a chain of balls $B_{r_j}(z_j)$, $j\in \{1,\ldots, m\}$, where $z_j\in \Omega$ and $m\le N\log (R_0/r)$, such that (see \cite[Theorem A.3]{CK2023})
$$
|y-y_0|\ge R_0/8, \quad x\in B_{r_1}(z_1), \quad y_0\in B_{r_m}(z_m), 
$$
$$
|z_j-y|\ge 16r_j, \quad r_j\le R_0, \quad j\in \{1,\ldots, m\},
$$
and 
$$
\Omega\cap B_{r_j}(z_j)\cap B_{r_{j+1}}(z_{j+1})\neq \emptyset, \quad j\in \{1,\ldots, m-1\}.
$$
In this case, we may assume that $r_m\ge R_0/(8\cdot 17)$.
Denote $T_j=\Omega\cap B_{r_j}(z_j)$.
For each $j\in \{1,\ldots, m-1\}$, we fix $\tilde{z}_j\in T_j\cap T_{j+1}$.
Then by \eqref{210802@eq3} we have 
$$
\begin{aligned}
\bigg|\dashint_{T_j}G(z, y)\,dz\bigg|
&\le \bigg|\dashint_{T_j} G(z, y)\,dz-G(\tilde{z}_j, y)\bigg|\\
&\quad +\bigg|\dashint_{T_{j+1}} G(z,y)\,dz-G(\tilde{z}_{j}, y)\bigg|+\bigg|\dashint_{T_{j+1}} G(z, y)\,dz\bigg|\\
&\le 2C_0+\bigg|\dashint_{T_{j+1}} G(z, y)\,dz\bigg|.
\end{aligned}
$$
This shows that 
$$
\bigg|\dashint_{T_1} G(z, y)\,dz\bigg|\le 2(m-1)C_0+\bigg|\dashint_{T_{m}} G(z, y)\,dz\bigg|\le 2(m-1)C_0+C,
$$
where we used \eqref{230209_eq6} in the second inequality.
Therefore, by \eqref{210802@eq3} again and the fact that 
$$
m\le C \log \bigg(\frac{\operatorname{diam}(\Omega)}{|x-y|}\bigg),
$$
we have 
$$
\begin{aligned}
|G(x,y)|
&\le \bigg|G(x,y)-\dashint_{T_1} G(z, y)\,dz\bigg|+\bigg|\dashint_{T_1} G(z, y)\,dz\bigg|\\
&\le 2mC_0+C\\
&\le C\log \bigg(\frac{\operatorname{diam}(\Omega)}{|x-y|}\bigg)+C.
\end{aligned}
$$
This proves \eqref{230209_eq5d}.

To complete the proof of the theorem, it remains to verify the symmetry property \eqref{230209_eq5e}.
Let $(G^*, \Pi^*)$ be the Green function for the adjoint operator $\cL^*$ in $\Omega$, which can be derived in the same manner.
For $x,y\in \Omega$ with $x\neq y$, we set $r=|x-y|/2$.
Notice that  $\eta G^*(\cdot,x)$ and $(1-\eta)G^*(\cdot,x)$ can be applied to \eqref{230205_eq1} as test functions, where  $\eta$ is a smooth function in $\bR^2$ satisfying 
$$
\eta\equiv 0 \, \text{ in }\, B_{r/2}(x), \quad \eta\equiv 1 \, \text{ in }\, \bR^2\setminus B_r(x).
$$
Since $G^*(\cdot,x)$ is continuous  in $\Omega\setminus \{x\}$ and  
$$
G^*(\cdot,x)=\eta G^*(\cdot,x)+(1-\eta)G^*(\cdot, x),
$$
by testing the $l$th columns of $\eta G^*(\cdot, x)$ and $(1-\eta)G^*(\cdot ,x)$ to \eqref{230206_eq1}, we have  
$$
\int_\Omega A^{\alpha\beta}_{ij}D_\beta G_{jk}(\cdot,y)D_\alpha G^*_{il}(\cdot,x)\,dz=G^*_{kl}(y,x).
$$
Similarly, we have 
$$
\int_\Omega A^{\alpha\beta}_{ij}D_\beta G_{jk}(\cdot,y)D_\alpha G^*_{il}(\cdot,x)\,dz=G_{lk}(x,y).
$$
This shows the equality \eqref{230209_eq5e}.
The theorem is proved.
\end{proof}

\bibliographystyle{plain}

%\bibliography{/Users/jongkeunchoi/Dropbox/R_Research/Z_Bibtex/Paper}

\end{document}